\documentclass[12pt]{amsart}
\usepackage[utf8]{inputenc}
\usepackage{setspace}
\usepackage[english]{babel}
\usepackage{mathabx}
\usepackage{framed}
\usepackage{soul}
\usepackage[all]{xy}
\usepackage[margin=1in]{geometry}
\usepackage{arydshln}
\usepackage[colorinlistoftodos]{todonotes}
\usepackage{pb-diagram}
\usepackage[mathscr]{euscript}
\usepackage{multicol}
\usepackage{stmaryrd} 
\usepackage{empheq} 
\usepackage{tikz-cd}
\tikzset{
  symbol/.style={
    draw=none,
    every to/.append style={
      edge node={node [sloped, allow upside down, auto=false]{$#1$}}}
  }
}

\usepackage{amsmath,amsfonts,amssymb,amsthm,mathrsfs,latexsym,mathtools}
\usepackage{commath}
\usepackage[T1]{fontenc}
\usepackage{hyperref}
\usepackage{caption} 
\usepackage{subcaption} 

\usepackage{lipsum}

\DeclareMathAlphabet{\mathbbmsl}{U}{bbm}{m}{sl}

\usetikzlibrary{matrix}
\title{Degeneration of Dual Varieties of Hypersurfaces}

\newtheorem{theorem}{Theorem}[section]
\newtheorem{definition}[theorem]{Definition}
\newtheorem{proposition}[theorem]{Proposition}

\newtheorem{example}[theorem]{Example}
\newtheorem{lemma}[theorem]{Lemma}

\newtheorem{thm}{Theorem}

\begin{document}
\author{Yilong Zhang}
\address{Department of Mathematics\\
Purdue University\\
  150 N University St, West Lafayette, IN 47907, USA}
\email{zhan4740@purdue.edu}
\date{Dec 29, 2023}
\subjclass[2010]{14D06, 14J70 primary}
\maketitle

\begin{abstract}
Consider a one-parameter family of smooth projective varieties $X_t$ which degenerate into a simple normal crossing divisor at $t=0$. What is the dual variety in the limit? We answer this question for a hypersurface of degree $d$ degenerate to the union of two hypersurfaces of degree $d_1$ and $d-d_1$ meeting transversely. We find all the irreducible components of the dual variety in the limit and their multiplicities.
\end{abstract}

\section{Introduction}
Consider a family of plane conics $C_t=\{xy+tz^2=0\}\subset \mathbb P^2$ with $t$ varying in a neighborhood of $0$. Then for $t\neq 0$, the dual curve of $C_t$ is a smooth conic, but when $t=0$, the conic consists of two lines, whose dual is the set of two points. The dual map on each individual fiber does not need to have the same dimension. However, if we look at the dual curves $C_t^*$ in family, which is $\{4tuv+w^2\}\subset (\mathbb P^2)^*$ and is a flat family. Take the limit as $t\to 0$, and we still get a degree two curve $\{w^2=0\}$. We prefer the second definition of the dual curve of the central fiber.


More generally, let $X_d$ be a degree $d\ge 2$ smooth hypersurface of $\mathbb P^{n+1}$ defined by $F_d$; $X_{d_1}$, $X_{d_2}$ be smooth hypersurfaces of degree $d_1$ and $d_2$ respectively defined by $F_{d_1}$ and $F_{d_2}$, with $d=d_1+d_2$. We require $F_d, F_{d_1},F_{d_2}$ to be general, so that their common zero locus is a complete intersection. Besides, by Bertini's theorem, $F^s:=sF_d+F_{d_1}F_{d_2}$ is smooth for $s\neq 0$ when $|s|$ is small enough. So for such $s\neq 0$, there is a dual map on smooth hypersurface $X^s:=\{F^s=0\}$ $$\mathcal{D}_s:X^s\mapsto (\mathbb P^{n+1})^*,$$
\begin{equation}
   x\mapsto \big (\frac{\partial F^s}{\partial x_0}(x),...,\frac{\partial F^s}{\partial x_{n+1}}(x)\big ),\label{dual_map}
\end{equation}
with $\frac{\partial F^s}{\partial x_j}(x)=s\frac{\partial F_d}{\partial x_j}(x)+F_{d_2}\frac{\partial F_{d_1}}{\partial x_j}(x)+F_{d_1}\frac{\partial F_{d_2}}{\partial x_j}(x)$, $j=0,...,n+1$.

The image $(X^s)^*$ is called the dual variety of $X^s$ and it is well known that it is a hypersurface of degree $m=d(d-1)^n$ \cite[Proposition 2.9]{3264}. So this defines a rational section $\mu$ on the sheaf $S^m(V^*)\otimes \mathcal{O}_{\Delta}$ over $\Delta$ which has possibly a pole along $s=0$ where $V=\mathbb C^{n+2}$. By multiplying by a suitable power of $s$, we can assume the section $\mu$ is regular and $\mu(0)\neq 0$. This defines a hypersurface at $s=0$ and will not change the nearby hypersurfaces.
\begin{definition}\normalfont
Let $(X^0)^*$ to be the projective hypersurface $\{\mu=0\}\subseteq (\mathbb P^{n+1})^*$ and call it the dual variety in the limit associated with the family $sF_d+F_{d_1}F_{d_2}=0$.
\end{definition}

Our purpose is to understand each irreducible component of $(X^0)^*$ and the multiplicity.

\begin{theorem}\label{dual_theorem} The dual variety in the limit $(X^0)^*$ associated to the family $\{F^s=0\}_{s\in\Delta}$ is a reducible hypersurface in $(\mathbb P^{n+1})^*$. Then $(X^0)^*$ consists of following components

\begin{itemize}
    \item[I.] dual varieties of $X_{d_1}$ and $X_{d_2}$, reduced;
    \item[II.] dual variety of $X_{d_1}\bigcap X_{d_2}$,   multiplicity two;
    \item[III.] dual variety of $X_d\bigcap X_{d_1}\bigcap X_{d_2}$, reduced.
\end{itemize}
\end{theorem}
When $n=1$, the component (III) is empty. When $d_i=1$, then the corresponding component $X_{d_i}^*$ is trivial (dual variety of a hyperplane is a point), and the components in (II) and (III) are cones over dual variety in the hyperplane $X_{d_i}$.

\begin{example}\normalfont
Consider the family of smooth cubic curves degenerate into a conic $Q$ union a line $L$ intersecting transversely, e.g., for an explicit example $$F^s(x_0,x_1,x_2)=s(x_0^3+x_1^3+x_2^3)+x_0(x_0^2+x_1^2+x_2^2)=0.$$

Let $(u,v,w)$ be the coordinates on dual space $(\mathbb P^2)^*$, then the dual variety in the limit $(X^0)^*$ consists of (1) a conic as the dual curve of $Q$; (2) lines $v=\pm iw$ with multiplicity two.

The dual curve of a cubic curve is a sextic curve and the degree decomposition reads $6=2+2\times 2$.
\end{example}

\begin{example}\normalfont
Consider a family of quintic surface $X$ degenerate to quadric surface $Q$ union a cubic surface $C$ meeting transversely. The dual variety in the limit $(X^0)^*$ consists of 
\begin{itemize}
    \item dual variety of $C$, which has degree 12
    \item dual variety of $Q$, which is again a quadric
    \item dual variety of a $(2,3)$-curve, which has degree 18 and multiplicity two
    \item dual variety of $30$ points, which is the union of 30 planes in $\mathbb P^3$.
\end{itemize}

The dual variety of a quintic surface has degree 80. The degree decomposition reads $80=12+2+2\times 18+30$.
\end{example}

\subsection*{Acknowledgement} I would like to thank my advisor Herb Clemens for the helpful discussions. I would like to thank Alexander Kuznetsov for answering my questions.


\section{Multiplicity Counting}
Assume $M,N$ are complex varieties and are flat over $\Delta$ via $f,g$. $h$ is a regular map making the diagram commute:

\begin{figure}[ht]
    \centering
\begin{tikzcd}
M \arrow[rr,"h"] \arrow[dr,"f"]&& N\arrow[dl,"g"]\\ &\Delta 
\end{tikzcd}
\end{figure}
 Let $Z$ be a component of $M_0=f^{-1}(0)$, let's assume that $h(Z)$ is a component of $N_0=g^{-1}(0)$. 

The multiplicity $m_Z$ of $Z$ is the order of vanishing $f^*t$ on $Z$, where $t$ is the local equation of $0\in \Delta$. Similarly, the multiplicity $n_{h(Z)}$ of $h(Z)$ is the order of vanishing of $g^*t$ on $h(Z)$. Then by $f^*=h^*\circ g^*$, we have the equality
\begin{equation}
    n_Z=k\cdot m_{h(Z)},\label{multiplicity}
\end{equation}
where $k$ is the ramification index of $h$ at the component $Z$, which can be defined as following:
Choose $p\in Z$ a general point, and $\Delta_p$ a holomorphic disk in $M$ which intersects $Z$ transversly at $p$, then the restriction $h|_{\Delta_p}$ is a $k$-to-$1$ map onto its image.

So we immediately have the following argument:

\begin{proposition}\label{multiplicity_determine_prop}
If $h$ has ramification index one along $Z$, then 
then $n_{h(Z)}$ coincides with $m_Z$.
\end{proposition}

\section{Proof of Theorem 1}
We define the total space of the family 
 \begin{equation}
      \mathscr{X}=\{sF_d+F_{d_1}F_{d_2}=0\}\subset \Delta\times \mathbb P^{n+1} \label{family_degree_d}
 \end{equation}
over the disk $\Delta$ with $\pi:\mathscr{X}\to \Delta$ the projection map. So the fiber over $s$ is the hypersurface $\{F^s=0\}$ and the special fiber $F^0=F_{d_1}F_{d_2}$ is reducible. The total space $\mathscr{X}$ is singular along $S:=\{s=0,F_d=0, F_{d_1}=0,F_{d_2}=0\}$ since it has local analytic equation $sx+yz=0$. There is a commutative diagram
\begin{figure}[ht]
    \centering
\begin{tikzcd}
\mathscr{X}\arrow[rr,dashed,"\mathcal{D}"] \arrow[dr,"\pi"]&& \Delta\times (\mathbb P^{n+1})^* \arrow[dl,"\pi_1"]\\ &\Delta, 
\end{tikzcd}
\end{figure}

\noindent where $\mathcal{D}:(s,p)\mapsto (\frac{\partial F^{s}}{\partial x_0}(p),...,\frac{\partial F^{s}}{\partial x_{n+1}}(p))$ is the dual map on each fiber, which is regular outside the locus $C:=\{s=0,F_{d_1}=0,F_{d_2}=0\}$.

Identify $S$ (resp. $C$) with the complete intersection $X_d\bigcap X_{d_1}\bigcap X_{d_2}$ (resp. $X_{d_1}\bigcap X_{d_2}$) in $\mathbb P^{n+1}$. We want to obtain a diagram as in the previous section. To do this, we need to resolve the singular locus of $\mathscr{X}$ and the indeterminacy locus of $\mathcal{D}$. So we blow up $\mathscr{X}$ along $S$ and then blow up the strict transform of the indeterminacy locus $C$ to get a smooth total space $\tilde{\mathscr{X}}$ with regular dual map $\tilde{\mathcal{D}}$, and reach a diagram

\begin{figure}[ht]
    \centering
\begin{tikzcd}
\tilde{\mathscr{X}} \arrow[d,"\lambda"] \arrow[dr,"\tilde{\mathcal{D}}"]\\ \mathscr{X}' \arrow[r,dashed,"\mathcal{D}'"] \arrow[d,"\pi'"]& \Delta\times (\mathbb P^{n+1})^*\arrow[dl,"\pi_1"]\\ \Delta, 
\end{tikzcd}
\end{figure}
\noindent where $\pi'$ is the composite of $\pi$ and the blowup $\sigma:\mathscr{X}'\to \mathscr{X}$. We first prove that

\begin{proposition}\label{Prop_mult=1}
$\tilde{D}$ has multiplicity index one on each component $\tilde{\mathscr{X}}^0$. 
\end{proposition}
\begin{proof}
This is because both $\pi$ and $\pi_1$ are projections to the first factors $\Delta$, moreover $\mathcal{D}$ is the identity map on this factor.
\end{proof}

\begin{proposition}\label{Prop_key}
The special fiber $\tilde{\mathscr{X}}^0:=(\lambda\circ\pi')^{-1}(0)$ has the following irreducible components: \\
I. strict transforms of $X_{d_1}$ and $X_{d_2}$, reduced;\\
II. the exceptional divisor $\tilde{C}$ over the strict transform of $C$, multiplicity two; \\
III. the exceptional divisor $\tilde{S}$ over $S$, reduced.\\
Moreover, their image under $\tilde{\mathcal{D}}$ are corresponding dual varieties of type I-III stated in Theorem \ref{dual_theorem}.
\end{proposition}

\begin{proof}
It suffices to show the special fiber has components of type II and III as described. 

The local analytic equation of $q_0\in C$ in $\mathscr{X}$ is $$u=0,\ v=0$$ in the hypersurface $\{sf(s,u,v)+uv=0\}\subseteq \Delta^3_{s,u,v}\times \Delta^{n-1}$, where $f=f(s,u,v)$ is an analytic function with $f(0,0,0)\neq 0$ and the last of $n-1$ variables are free. If $q_0\in C\setminus S$, then it has a neighborhood unaffected by the first blowup $\sigma$, so the multiplicity of  $\tilde{C}$ is the same as the multiplicity of the exceptional divisor of blowup of $(0,0,0)$ of $\{sf(s,u,v)+uv=0\}\subseteq \Delta^3_{s,u,v}$, which is \textit{two} by straightforward computation.

Now let's prove the image of $\tilde{C}$ under $\tilde{\mathcal{D}}$ is the dual variety of $X_{d_1}\bigcap X_{d_2}$. The total transform of $\lambda$ has local analytic equation over a point $q_0\in C\setminus S$ 

\begin{figure}[ht]
\centering
$sf(s,u,v)+uv=0\ ,\textup{rank}\begin{bmatrix} 
\alpha & \beta &\gamma\\
u & v &s\\
\end{bmatrix}\le 1,\ (\alpha,\beta,\gamma)\in \mathbb P^2.$
\end{figure}

Set $\beta=1$, then $u=\alpha v$, $s=v\gamma$, $v(\gamma f(s,u,v)+\alpha v)=0$, and the dual map becomes
\begin{equation}
    \tilde{\mathcal{D}}:(\gamma,v,\alpha)\mapsto \big (\gamma\frac{\partial F_{d}}{\partial x_j}(q)+\frac{\partial F_{d_1}}{\partial x_j}(q)+\alpha\frac{\partial F_{d_2}}{\partial x_j}(q)\big )_{j=0,...,n+1}\in (\mathbb P^{n+1})^*,\label{dual_map_ext}
\end{equation}
where the local coordinate of $q$ depends on $\gamma,v,\alpha$. As $v$ goes to zero, by equation $\gamma f+\alpha v=0$, $\gamma$ goes to $0$ and we have the dual map on the exceptional divisor

$$\tilde{\mathcal{D}}(q_0)=\big (\frac{\partial F_{d_1}}{\partial x_j}(q_0)+\alpha\frac{\partial F_{d_2}}{\partial x_j}(q_0)\big )_{j=0,...,n+1}\in (\mathbb P^{n+1})^*.$$

It follows that the image at $q_0\in X_{d_1}\bigcap X_{d_2}\setminus X\bigcap X_{d_1}\bigcap X_{d_2}$ is the set of linear combination of normal vectors along $F_{d_1}$ and $F_{d_2}$ at $q_0$, which form a Zariski open dense subset of $(X_{d_1}\bigcap X_{d_2})^*$. This finishes the proof of case II.

In case III (only exist when $n\ge 2$), we choose a point $p_0\in S$, then $S$ has local analytic equation in $\mathscr{X}$ $$x=0,\ y=0,\ z=0$$ in $\{sx+yz=0\}\subseteq \Delta^4_{s,x,y,z}\times \Delta^{n-2}$ where the last $n-2$ variables are free. Since the exceptional divisor $\tilde{S}=\sigma^{-1}(S)$ is generically unaffected by the second blowup $\lambda$, the multiplicity of $\tilde{S}$ is computed over a general point $p_0\in S$. Again, this multiplicity coincides with the multiplicity of the exceptional divisor of the blowup of origin of $sx+yz=0$, which is \textit{one}.

Finally, let's show that the image of $\tilde{S}$ under $\tilde{\mathcal{D}}$ is the dual variety of $X_d\bigcap X_{d_1}\bigcap X_{d_2}$. (By abuse of notation, $\tilde{S}$ is also denoted as its strict transform under the blowup $\lambda$.) The total transform of $\sigma$ has local equation over a general point $p_0\in S$ 
\begin{gather}
sx+yz=0,\ \textup{rank}\begin{bmatrix} 
\alpha & \beta &\gamma &\delta\\
s & x &y &z \\
\end{bmatrix}\le 1,\ (\alpha,\beta,\gamma,\delta)\in \mathbb P^3.\label{quadric_transform}
\end{gather}

If we choose an affine chart $\alpha=1$, then the equation \eqref{quadric_transform} becomes $x=s\beta,y=s\gamma,z=s\delta$ and $s^2(\beta+\gamma\delta)=0$, so by substitution and scaling in projective coordinate,  the dual map \eqref{dual_map} becomes
\begin{equation}
    \mathcal{D}':(s,\gamma,\delta)\mapsto \big (\frac{\partial F_d}{\partial x_j}(p)+\gamma \frac{\partial F_{d_1}}{\partial x_j}(p)+\delta \frac{\partial F_{d_2}}{\partial x_j}(p)\big )_{j=0,...,n+1}\in (\mathbb P^{n+1})^*\label{dual_map_affine}
\end{equation}
around a point $p_0\in S$ and the local coordinate of $p$ depends on $s,\gamma,\delta$. So \eqref{dual_map_affine} implies that the dual map extends to (a Zariski open subset of) $\sigma^{-1}(S)$ by
\begin{equation}
    S\times \mathbb C^2\ni(p_0,\gamma,\delta)\mapsto \big (\frac{\partial F_d}{\partial x_j}(p_0)+\gamma \frac{\partial F_{d_1}}{\partial x_j}(p_0)+\delta \frac{\partial F_{d_2}}{\partial x_j}(p_0)\big )_{j=0,...,n+1}\in (\mathbb P^{n+1})^*.\label{dual_map_on_blowup}
\end{equation}

Note the map is well-defined on this chart due to the assumption that $S$ is smooth complete intersection, so three normal directions $\frac{\partial F_{d}}{\partial x_j},\frac{\partial F_{d_1}}{\partial x_j},\frac{\partial F_{d_2}}{\partial x_j}$ are linearly independent. This shows that the image of $\tilde{S}$ under $\tilde{\mathcal{D}}$ contains a Zariski dense subset of $(X_d\bigcap X_{d_1}\bigcap X_{d_2})^*$, so it has to be the whole dual variety.
\end{proof}

Consequently, the images of type I, II, and III in Proposition \ref{Prop_key} are the corresponding components in Theorem \ref{dual_theorem}. It remains to show there are no other components.

Since the dual varieties $\{(X^s)^*\}_{s\in \Delta}$ as we defined in the previous section is a flat family, in particular, each member has the same degree. In what follows, we will show that the sum of the degree of the components of three types agrees with the degree of the nearby fiber. This proves that $(X^0)^*$ has no other components and, therefore will complete the proof of Theorem \ref{dual_theorem}.

\begin{proposition}\label{degree_identity_corollary}
The following identities hold:
\begin{align}
\deg(X_d^*)=\deg(X_{d_1}^*)+\deg(X_{d_2}^*)+2\deg((X_{d_1}\bigcap X_{d_2})^*), \text{if}\ n= 1;\\
\deg(X_d^*)=\deg(X_{d_1}^*)+\deg(X_{d_2}^*)+\deg((X_d\bigcap X_{d_1}\bigcap X_{d_2})^*)+2\deg((X_{d_1}\bigcap X_{d_2})^*), \text{if}\ n\ge 2. \nonumber
\end{align}
\end{proposition}

\begin{proof}
First of all, the degree of dual variety of hypersurface is well known \cite[Proposition 2.9]{3264}. In particular, $X_d^*$ has degree $d(d-1)^{n}$, and $X_{d_i}^*$ has degree $d_i(d_i-1)^{n}$, $i=1,2$. So $n=1$ case is the consequence of the identity
$$d(d-1)=d_1(d_1-1)+d_2(d_2-1)+2d_1d_2,$$
which one can readily check by hand.

For $n\ge 2$ case, we need a formula for the dual variety of a complete intersection. According to Kleiman  \cite[p362]{Kleiman}, if $Y\subset \mathbb P^N$ is a smooth variety of dimension $m$ with dual $Y^*$ being a hypersurface, then the degree of the dual variety $\deg(Y^*)$ equals $\int_Ys_m(E)$, where $s_m$ is the Segree class, $E$ is the vector bundle $N^*_{Y|\mathbb P^N}\otimes \mathcal{O}_Y(1)$ over $Y$, and 
 $N^*_{Y|\mathbb P^N}$ is the conormal bundle of $Y$. So we have 

\begin{lemma} Let $Y$ be a complete intersection of type $(d_1,...,d_k)$ of dimension $n$. Then the degree of the dual variety $\deg(Y^*)$ coincides with the coefficient of $h^n$ of the polynomial

$$\prod_{i=1}^k(1-(d_i-1)h)^{-1}\prod_{i=1}^kd_i.$$
\end{lemma} 
\begin{proof}
  In this case, $E=\oplus_i\mathcal{O}_Y(1-d_i)$. So the Chern polynomial is $c(E)=\prod_i(1+(1-d_i)h)$, where $h$ is the hyperplane class. Then the result follows from the fact that $s(E)=c^{-1}(E)$ and $\int_Yh^n=\deg(Y)=\prod_id_i$.
\end{proof}
Apply the formula to $X_d\bigcap X_{d_1}\bigcap X_{d_2}$, which has dimension $n-2$, so we get its degree of dual variety 
$$N_{d,d_1,d_2}^{n-2}:=\deg((X_d\bigcap X_{d_1}\bigcap X_{d_2})^*)=\sum_{\substack{i+j+k=n-2\\i,j,k\ge 0}}(d_1-1)^i(d_2-1)^j(d-1)^kd_1d_2d.$$  Similarly to the complete intersection $X_{d_1}\bigcap X_{d_2}$ of dimension $n-1$, the degree of its dual variety is
$$N_{d_1,d_2}^{n-1}:=\deg((X_{d_1}\bigcap X_{d_2})^*)=\sum_{\substack{i+j=n-1\\i,j\ge 0}}(d_1-1)^i(d_2-1)^jd_1d_2.$$

So the proof of Corollary \ref{degree_identity_corollary} reduces to show the identity
\begin{equation}\label{equality}
    d(d-1)^n=d_1(d_1-1)^n+d_2(d_2-1)^n+N_{d,d_1,d_2}^{n-2}+2N_{d_1,d_2}^{n-1}.
\end{equation}

We prove by induction on $n$. The base case is $n=2$, one readily checks the following identity holds
$$d(d-1)^2=d_1(d_1-1)^2+d_2(d_2-1)^2+dd_1d_2+2(d_1+d_2-2)d_1d_2.$$

We want to show the equality \eqref{equality}. Assuming the identity \eqref{equality} holds for $n-1$ case, we have 
\begin{align*}
    d(d-1)^n&=d(d-1)^{n-1}(d-1)=[d_1(d_1-1)^{n-1}+d_2(d_2-1)^{n-1}+N_{d,d_1,d_2}^{n-3}+2N_{d_1,d_2}^{n-2}](d-1)\\
    &=d_1(d_1-1)^n+d_2(d_2-1)^n+d_1d_2[(d_1-1)^{n-1}+(d_2-1)^{n-1}]+(d-1)(N_{d,d_1,d_2}^{n-3}+2N_{d_1,d_2}^{n-2}).
\end{align*}

So it reduces to show the identity 
$$N_{d,d_1,d_2}^{n-2}+2N_{d_1,d_2}^{n-1}=(d-1)(N_{d,d_1,d_2}^{n-3}+2N_{d_1,d_2}^{n-2})+d_1d_2[(d_1-1)^{n-1}+(d_2-1)^{n-1}],$$
which is a consequence of identity $(d-1)N_{d,d_1,d_2}^{n-3}=N_{d,d_1,d_2}^{n-2}-dN_{d_1,d_2}^{n-2}$ plus direct computations. This finishes the proof.
\end{proof}

\bibliographystyle{plain}
\bibliography{bibfile}

\end{document}